\documentclass[12pt]{amsart}
\usepackage{a4ams}
\usepackage{amsfonts}
\usepackage{amssymb}
\usepackage{amsmath}
\usepackage[arrow,matrix,curve]{xy}

\usepackage[active]{srcltx}

\usepackage{color}

\usepackage[normalem]{ulem}
\usepackage{color}

\newcommand{\C}{{\mathbb C}}

\newcommand{\E}{{\mathbb E}}

\newtheorem{theorem}{Theorem}[section]
\newtheorem{lemma}[theorem]{Lemma}

\newtheorem{corollary}[theorem]{Corollary}
\newtheorem{proposition}[theorem]{Proposition}
\newtheorem{definition}[theorem]{Definition}

\def\cal{\mathcal}

\newcommand{\calf}[0]{{\cal F}}
\newcommand{\call}[0]{{\cal L}}
\newcommand{\calg}[0]{{\cal G}}

\newcommand{\cale}[0]{{\cal E}}
\newcommand{\calk}[0]{{\cal K}}
\newcommand{\calo}[0]{{\cal O}}

\newcommand{\calp}[0]{{\cal P}}

\newcommand{\caln}[0]{{\cal N}}

\begin{document}               

\title[Equivariant $KK$-theory]{Equivariant $KK$-theory of $r$-discrete groupoids and inverse semigroups }
\author[B. Burgstaller]{Bernhard Burgstaller}
\address{Doppler Institute for mathematical physics,
Trojanova 13, 12000 Praha, Czech Republic}
\email{bernhardburgstaller@yahoo.de}
\thanks{This work was supported by Czech MEYS Grant LC06002.}

\begin{abstract}
For an $r$-discrete Hausdorff groupoid $\calg$ and an inverse
semigroup $S$ of slices of $\calg$ there is an isomorphism between
$\calg$-equivariant $KK$-theory and compatible $S$-equivariant $KK$-theory.
We use it to define descent homomorphisms for $S$, and indicate a
Baum--Connes map for inverse semigroups.
Also findings by Khoshkam and Skandalis for crossed products by inverse semigroups
are reflected in $KK$-theory.
\end{abstract}

\maketitle

\section{Introduction}


Let $\calg$ be an $r$-discrete Hausdorff groupoid and $S$ a full system of slices of $\calg$, that is,
$S$ is an inverse semigroup.
In such a setting one seems to start with a groupoid, but note that on the other hand
an arbitrary inverse semigroup $S$ can be embedded as a full system of slices
in an $r$-dicrete groupoid $\calg_S$, which however need not be Hausdorff,
by a construction by Paterson \cite{paterson1999}.
In \cite{0992.46051}, Quigg and Sieben showed that for a $C^*$-algebra $A$
there is a one-to-one correspondence between a $\calg$-action on $A$
and an $S$-action on $A$ (in Sieben's sense \cite{sieben1997}).
Moreover, they found an isomorphism
between the corresponding groupoid crossed product $A \rtimes \calg$
and the corresponding inverse semigroup crossed product $A \widehat \rtimes S$
(in Sieben's sense \cite{sieben1997}).
Khoshkam and Skandalis \cite{1061.46047} introduced
another crossed product $A \rtimes S$ for $S$
which usually significantly differs from Sieben's crossed product,
but have shown in \cite{1061.46047} that $(A \rtimes E) \widehat{\rtimes} S$
and $A \rtimes S$ are isomorphic, where $E$ denotes the set of idempotent elements of $S$.

In this note we aim to extend these results to a $KK$-theoretical level.
We use semimultiplicative set equivariant $KK$-theory
$KK^S$ introduced in \cite{burgiSemimultiKK}, but also slightly adapt it such that the underlying $C_0(X)$-structure
is compatible with the module multiplication and denote this so-called compatible equivariant $KK$-theory by $\widehat{KK^S}$.

We then use Quigg and Sieben's findings for an isomorphism $\rho$
between $\calg$-equivariant $KK$-theory $KK^\calg(A,B)$ and compatible $S$-equivariant
$KK$-theory $\widehat{KK^S}(A,B)$,
see Theorem
\ref{theoremKKQuiggSieben}.
Khoshkam and Skandalis' findings will be reflected by a
homomorphism $\epsilon:\widehat{KK^S} (A,B) \longrightarrow \widehat{KK^S}(A \rtimes E, B \rtimes
E)$ which we call an expansion homomorphism and which is
defined like a descent homomorphism, see Theorem
\ref{theoremExpansionHomomorphismHat}.
We use $\rho$ and $\epsilon$
to define and intertwine, respectively, three types of descent homomorphisms for $\widehat{KK^S}$, one,
$\widehat{j^S}$, of Sieben's
crossed product-type, and two, $j^S$ and $j^S_r$, of
Khoshkam and Skandalis' crossed product-type
by pulling back corresponding descent homomorphisms from groupoid equivariant $KK$-theory to inverse
semigroup equivariant $KK$-theory. 
Both the isomorphism $\rho$ and the expansion homomorphism
$\epsilon$ respect functoriality and the Kasparov product. The isomorphism $\rho$ intertwines the descent homomorphisms $\widehat{j^S}$
and $j^\calg$, and $\epsilon$ the descent homomorphisms
$\widehat{j^S}$ and $j^S$.

%
%

An application of our translation, Theorem
\ref{theoremKKQuiggSieben}, is that the Baum--Connes maps \cite{baumconneshigson1994}
associated to $\calg$ \cite{0932.19005}
readily yield Baum--Connes maps associated to $S$,
see Corollaries \ref{corollaryBCmapSieben} and \ref{corollaryBCmapKhoshkamSkandalis}.
Many $C^*$-algebras like the Cuntz--Krieger algebras
\cite{cuntzkrieger}
are obviously generated by inverse semigroups and are potentially inverse semigroup crossed products,
but are less obvious groupoid crossed products. Also, inverse semigroups and their crossed products seem simpler at first glance,
so an inverse semigroup Baum--Connes theory may potentially be useful.

The overview of this note is as follows.
In section \ref{SectionPreliminaries} we recall some facts about $r$-discrete groupoids,
inverse semigroups, $C_0(X)$-algebras and their bundle picture, and the corresponding $KK$-theories.
In Section \ref{SectionCompatibleKK} we define compatible $S$-equivariant $KK$-theory,
the $S$-equivariant $KK$-theory we mainly use in this note.
In Section \ref{SectionCrossed} we very briefly recall some facts about the various
crossed products existing for inverse semigroups.
In Section \ref{sectionIso} we show that $\rho$ is an isomorphism,
and in Section \ref{sectionExpansion} we define and consider the expansion
homomorphism $\epsilon$.

\section{Preliminaries}    \label{SectionPreliminaries}

Throughout, $\calg$ denotes a locally compact $r$-discrete
Hausdorff groupoid endowed with the
Haar system of counting measures and $X$ its base space.
Since slices play a crucial role here, we recall the definition.

\begin{definition}[Renault \cite{renault1980}]   \label{definitionSlice}
{\rm
A {\em slice} of a locally compact Hausdorff groupoid $\calg$ is an open subset $\calo$ of $\calg$ such that the range and the source maps $r$ and $s$, respectively,
when restricted to $\calo$, become homeomorphisms, i.e. $s|_\calo:\calo \longrightarrow s(\calo)$ and $r|_\calo:\calo \longrightarrow r(\calo)$ are homeomorphisms.
$\calg$ is called {\em $r$-discrete} if the set of all slices of $\calg$ is a basis for the topology of $\calg$.
The set of all slices of $\calg$ forms an inverse semigroup under the set operations
$\calo \cdot \calp=\{gh \in \calg|\, g\in \calo, \, h \in \calp\}$ and
$\calo^* =\{g^{-1} \in \calg|\, g \in \calo\}$, where $\calo$ and $\calp$ denote two slices.
An inverse semigroup of slices of $\calg$ is called {\em full} if it generates the topology of $\calg$.
}
\end{definition}

Throughout, $S$ denotes an inverse semigroup and $E$ its set of projections.
We assume that $S$ is upward-directed (because this is required in
Quigg and Sieben \cite{0992.46051},
and this note builds on \cite{0992.46051}, but we do not need this property explicitly in this note).
In other words, for every $e_1,e_2 \in E$ there exists $e_3 \in E$ such that
$e_1,e_2 \le e_3$.
%
The following definition sums up our setting.

\begin{definition}
{\rm
A {\em Cantor inverse semigroup $S$ embedded in $\calg$} is an
upward-directed inverse semigroup $S$ which is a full system of slices of an $r$-discrete Hausdorff
groupoid $\calg$ such that every projection $e \in E$ is a {\em clopen} subset of the base space $X$ of $\calg$.
Multiplication and inversion in $S$ is understood to come from the corresponding
operations defined for slices (see Definition \ref{definitionSlice}).
}
\end{definition}

Throughout, if nothing else is said, $S$ denotes a Cantor inverse semigroup embedded in $\calg$.
Note that by the clopen-condition $X$ is a totally disconnected space.
Given a clopen subset $Y$ of $X$,
$1_Y \in C_0(X)$ denotes
the continuous characteristic function of $Y$ on $X$.
It is clear that
$C_0(X)$ is the norm closure of the linear span of the family of characteristic functions $(1_e)_{e \in E}$.
If $s \in S$ and $x \in X$ then we write $s x$ for the unique single element $g \in s \subseteq \calg$ such that $s(g) = x$.
Analogously we define $x s$.

We use inverse semigroup equivariant $KK$-theory $KK^S$ as introduced
in \cite{burgiSemimultiKK} by regarding $S$ as a semimultiplicative set.
The definition of $KK^S$ is also summarized in Section 3 of \cite{burgiDescent}, and we prefer using it as our reference.

The evaluation of an $S$-action $\alpha$ on an {$S$-Hilbert $C^*$-algebra} $A$ (\cite[Def. 3.1]{burgiDescent}) is usually denoted by $s(a)$ rather than
$\alpha_s(a)$ ($s \in S$ and $a \in A$). The restriction to $E$ is an action $\alpha|_E : E \longrightarrow Z(\call(A))$ into
the center of the multiplier algebra of $A$, see last section of \cite{burgiSemimultiKK} or \cite[Lemma 5.10]{burgiDescent}.
Consequently, we have the important relation $e(a) b = a e(b) = e(ab)$ for $e \in E$ and $a,b \in A$,
see \cite[Lemma 5.8.(iii)]{burgiDescent}.

The $S$-action on a $S$-Hilbert module $\cale$ (\cite[Def. 3.3]{burgiDescent}) is usually denoted by $U$.
Note that $U_s^* = U_{s^*}$ by \cite[Cor. 4.6]{burgiDescent}. We often write $s(T)$ for $U_s T U_s^*$ when $T \in \call(\cale)$.  

\begin{definition}
{\rm
An $S$-equivariant representation $\pi: A \longrightarrow \call(\cale)$ (\cite[Def. 3.4]{burgiDescent}) is called {\em compatible} 
if $\pi(e(a))= e(\pi(a))$ for all $e \in E$ and $a \in A$.
}
\end{definition}

It is not difficult to check that an $S$-equivariant representation $\pi$ is compatible if and only if $\pi(s(a)) = U_s \pi(a) U_s^*$ for all $s \in S$ and $a \in A$.
Indeed, if $\pi$ is compatible then $\pi(s(a)) = \pi(s s^* s(a)) =
\pi(s(a)) U_s U_s^* = U_s \pi(a) U_s^*$ by \cite[Def. 3.4]{burgiDescent}.

Our basic reference for $C_0(X)$-Banach spaces, -algebras and
-modules for a topological space $X$ is Paravicini's thesis
\cite{ParaviciniThesis}, and also Le Gall \cite{legall1999} who is
specific about $C^*$-algebras and Hilbert modules. Recall that
Hilbert modules inherit their $C_0(X)$-structure from the
$C^*$-algebra \cite{legall1999}. A $C_0(X)$-Banach space $\cale$
is called {\em locally $C_0(X)$-convex} \cite[Def.
4.4.1]{ParaviciniThesis} if $\|\xi\| = \sup_{x \in X} \|\xi_x\|$
for all $\xi \in \cale$. There is an equivalence of
categories between the category of $C_0(X)$-convex $C_0(X)$-Banach spaces together with bounded $C_0(X)$-linear maps as morphisms and
the category of upper semi-continuous fields of Banach spaces over $X$ together with bounded continuous fields of linear maps
as morphisms, see
\cite{ParaviciniThesis} (see Definition 4.2.8, Section 4.4 and
Appendix A.2.2 in \cite{ParaviciniThesis}). Since
$C_0(X)$-$C^*$-algebras and Hilbert modules over
$C_0(X)$-$C^*$-algebras automatically satisfy $C_0(X)$-convexity,
the last equivalence applies to them.
%
In this note we will mainly work with the equivalent
field picture without much comment.

Let $f:Y \longrightarrow X$ be a continuous function and $\cale$ an upper semi-continuous field of Banach spaces over $X$.
By \cite[Definition 3.3.1]{ParaviciniThesis}, the
{\em pullback} $f^* \cale$ is the set of all (set-theoretical) sections in $\prod_{y \in Y} \cale_{f(y)}$
which are in every point $y_0 \in Y$ locally approximable (see \cite[Definition 3.1.13]{ParaviciniThesis})
by a simple function of the set
$\cale \circ f = \{\xi \circ f| \, \xi \in \cale\}$.
The set of simple functions is called a {\em total set} of $f^* \cale$, see \cite[Proposition 3.1.26]{ParaviciniThesis}.



In the next lemma we recall that any inverse semigroup can be realized as a system of slices of a groupoid
by using Paterson's construction in \cite{paterson1999}.

\begin{lemma}    \label{lemmaPatersonsGroupoidCX}
Let $S$ be an upward-directed inverse semigroup and suppose that Paterson's groupoid $\calg_S$ associated to $S$ is Hausdorff. Then
$S$ is a Cantor inverse semigroup embedded in $\calg_S$,
and $\varphi:C^*(E) \longrightarrow C_0(X)$, where $\varphi(e)=1_e$
for every $e \in E$, is an isomorphism.
(Here, $C^*(E)$ denotes the universal abelian
$C^*$-algebra generated by the projection set $E$.)
Moreover, every full $S$-Hilbert $C^*$-algebra $A$ (i.e. $C_0(X)A=A$) is automatically a compatible $S$-Hilbert $C^*$-algebra.
\end{lemma}


\begin{proof}
See Paterson \cite{paterson1999} (or
\cite[p. 64]{khoshkamskandalis2002}) for more on this.
For the last claim note that by universality of $C^*(E)$ the
$E$-action $\alpha|_E: E \longrightarrow Z(\call(A))$ on a $S$-Hilbert
$C^*$-algebra $A$ extends automatically to a $*$-homomorphism $\tilde \alpha: C^*(E) \longrightarrow
Z(\call(A))$. So automatically, $A$ becomes a $C_0(X)$-algebra.
\end{proof}

\section{Compatible $S$-equivariant $KK$-theory}

\label{SectionCompatibleKK}
%

$S$-equivariant $KK$-theory and $\calg$-equivariant $KK$-theory are not compatible.
In order to get an isomorphic theory we need to slightly adapt $S$-equivariant $KK$-theory.
Assume that all $S$-$C^*$-algebras are also $C_0(X)$-algebras. We need to claim that the $C_0(X)$-action,
when restricted to the characteristic functions of elements of $E$, coincides with the $E$-action.

\begin{definition}  \label{DefCompatibleHilbertAlgebra}
{\rm
Call an $S$-Hilbert $C^*$-algebra $A$ {\em compatible} if $1_e \cdot a = e(a)$ for all $e \in E$ and $a \in A$.  
}
\end{definition}

\begin{definition}  \label{definitionCompatibleHilbert}
{\rm
Call an $S$-Hilbert module $\cale$ over a compatible $S$-Hilbert $C^*$-algebra $B$ {\em compatible}
if $U_e(\xi) b = \xi e(b)$ for all $e \in E, \xi \in \cale$ and $b \in B$. 
}
\end{definition}

Similarly, an $S$-Hilbert $(A,B)$-bimodule $\cale$ is called compatible if also the $A$-module is compatible in the sense that
$e(a) \xi = a U_e(\xi)$. Of course, this is equivalent to saying that the $S$-equivariant map $\pi:A \longrightarrow \call(\cale)$
of $(\cale,U,\pi)$
is compatible.
%
The compatibility of Hilbert modules implies compatibility with the $C_0(X)$-structure.
We have $\xi (1_e \cdot b) = (1_e \cdot \xi) b$ and $(1_e \cdot a) \xi = a (1_e \cdot \xi)$ for all $e \in E$. By the fact that $C_0(X)$ is generated by
these $1_e$'s, this shows that the $C_0(X)$-structure is compatible with the module multiplications, exactly as we have
it in groupiod-equivariant Hilbert modules in the sense of Le Gall \cite{legall1999}.

\begin{definition}   \label{definitionInverseSemigroupSlicesKK}
{\rm For compatible $S$-Hilbert $C^*$-algebras $A$ and
$B$, we write $\widehat{\E^S}(A,B)$ for the collection of all cycles
$(\cale,T)$ in $\E^G(A,B)$ (\cite[Def. 3.6]{burgiDescent}) where $\cale$ is
a compatible $S$-Hilbert $(A,B)$-bimodule.
Compatible $S$-equivariant $KK$-theory $\widehat{KK^S}(A,B)$ is defined to be $\widehat{\E^S}(A,B)$
divided by homotopy induced by $\widehat{\E^S}(A,B[0,1])$. }
\end{definition}

The hat over $KK^S$ should indicate that the underlying Hilbert modules are compatible.
In the remainder of this section assume that all modules and bimodules are compatible.
Given $S$-Hilbert modules $\cale_1$ and $\cale_2$ we define the internal tensor product
$\cale_1 \otimes_B \cale_2$ as in \cite{burgiSemimultiKK}.
Since $B$ acts compatibly on $\cale_2$, the internal tensor product
is automatically $C_0(X)$-balanced.
The exterior tensor product $\cale_1 \otimes \cale_2$ we define $C_0(X)$-balanced
as in Le Gall \cite[Definition 4.2]{legall1999}).
That the diagonal $S$-action is indeed well defined
can be seen as follows.
One has
$$U_s(1_e \xi) = U_s U_e (\xi) = U_{s e s^*}
U_s (\xi) = 1_{s e s^*} U_s(\xi)$$
for $s \in S, e \in E$ and $\xi \in
\cale$, and so $U^{(1)}(1_e \xi) \otimes U^{(2)}(\eta)= U^{(1)}(\xi) \otimes U^{(2)}(1_e \eta)$
in $\cale_1 \otimes \cale_2$ for $\eta \in \cale_2$.
Both the exterior and internal tensor product are compatible Hilbert modules.

For a separable compatible $S$-Hilbert $C^*$-algebra
$D$ we define a map
$\widehat{\tau_D^{S}}: \widehat{KK^S}(A,B)
\longrightarrow \widehat{KK^S}(A  \otimes D,B
 \otimes D)$
by $\widehat{\tau_D^{S}}
[\cale,T] = [\cale \otimes D, T \otimes 1]$
for $(\cale,T) \in \widehat{\E^S}(A,B)$.
This map is completely analog to the corresponding well known map $\tau^\calg_D$
in groupoid equivariant $KK$-theory.

\begin{lemma}
There exists a Kasparov product for
$\widehat{KK^S}$ which is defined on the level on cycles analogously like the Kasparov product for $KK^S$ in \cite{burgiSemimultiKK}.
There exists a Kasparov cup-cap product for $\widehat{KK^S}$
which is defined as usual by using $\widehat{\tau^S}$.
\end{lemma}

\begin{proof}
The Hilbert module part of a cycle of a Kasparov product is $\cale \otimes_{B} \calf$, which is compatible
if $\cale$ and $\calf$ are compatible, so the Kasparov product from \cite{burgiSemimultiKK} is also well defined
for $\widehat{KK^S}$.
If one wants check associativity of the cup-cap product for $\widehat{KK^S}$
similarly as in the last Section 7 of
\cite{burgiSemimultiKK}, say, then one has to use $\widehat{\tau^S}$ rather
than $\tau^S$ and the one element $1 = [{C_0(X)},0] \in
KK^S(C_0(X),C_0(X))$ instead of $[\C,0] \in
KK^S(\C,\C)$.
Alternatively, and simpler, than going through these proofs, is to use the isomorphism between $KK^S$ and $KK^\calg$ of Theorem
\ref{theoremKKQuiggSieben} together with the evident associativity of the groupoid equivariant Kasparov product.
\end{proof}

\section{Crossed products}

\label{SectionCrossed}

For crossed products of inverse semigroups
we will use constructions from three sources: Sieben's full
crossed product \cite{sieben1997}, Khoshkam and Skandalis'
reduced and full crossed products \cite{1061.46047}, and the
author's full crossed product \cite{burgiDescent}.
In any case, an $S$-action is an inverse semigroup homomorphism
from $S$ to some objects related to the $C^*$-algebra.
Khoshkam and Skandalis \cite{1061.46047} are most general and allow morphisms from $S$
into the isomorphisms between quotients of ideals, Sieben \cite{sieben1997} into
partial automorphisms, and the author \cite{burgiDescent} into endomorphisms which are
also partial isometries (see Section 4 of \cite{burgiDescent}), the smallest class.

For a $C_0(X)$-algebra $A$ we write $A_e := 1_e \cdot A$ for $e \in E$.
$A$ is a $C^*$-dynamical system $(A,\beta,S)$ in Sieben's sense if $\beta$ is an $S$-action by partial automorphisms 
compatible with the $C_0(X)$-structure, that is, the domain of $\beta_e$ is the ideal
$A_e$ of $A$ for every $e \in E$.
Here is the key why we require $e \in E$ to be a clopen set in $X$; we need to have that $1_e$ is a continuous
function in $C_0(X)$.


\begin{lemma}
Let $A$ be a $C_0(X)$-algebra and $S$ a Cantor inverse semigroup. Then $(A,\alpha,S)$ is a compatible $S$-Hilbert $C^*$-algebra if and only if
$A$ is a $C^*$-dynamical system $(A,\beta,S)$ in Sieben's sense. 
\end{lemma}

\begin{proof}
If $\alpha$ is a compatible $S$-action then $\beta=(\beta_s:A_{s^*s} \rightarrow A_{s s^*})_{s \in S}$,
$\beta_s:= \alpha_s|_{A_{s^*s}}$, defines a partial action in Sieben's sense on $A$ which is compatible
with the $C_0(X)$-structure. In the other direction, given an action $\beta$,
$\alpha_s(a):= \beta_s(1_{s^* s} \cdot a)$ ($s \in S$ and $a \in A$) defines a compatible
$S$-action on $A$.
\end{proof}


The author's full
crossed product from \cite{burgiDescent} coincides with the one of Khoshkam--Skandalis by
\cite[Corollary 8.4]{burgiDescent}, and is defined as follows.
%
%
%
%
The crossed product $A \rtimes S$ is the enveloping $C^*$-algebra of
the Banach $*$-algebra $\ell^1(S,A)$. The latter is the $\ell^1$-norm
closure of the set of formal sums $\sum_{s \in S} a_s \rtimes s$
with finite support,
where  $a_s \in A_{s s^*}$ for all $s \in S$, endowed with the
natural convolution product, involution and $\ell^1$-norm (see
\cite[Lemma 5.13]{burgiDescent}), that is, $(a_s \rtimes s)^* = s^*(a_s^*) \rtimes s^*$
and $(a_s \rtimes s) (b_t \rtimes t) = a_s s(b_t) \rtimes st$ for $s,t \in S, a_s \in A_{s s^*}$ and $b_t
\in A_{t t^*}$.
The crossed product in
Sieben's sense, denoted by $A \widehat \rtimes S$, is the
enveloping $C^*$-algebra of $\ell^1(S,A)$ with respect to compatible
$S$-equivariant representations $(\pi, u,H)$ on Hilbert spaces $H$.

\section{The isomorphism in $KK$-theory}   \label{sectionIso}

Throughout this section all $C^*$-algebras are $C_0(X)$-algebras and all $S$-Hilbert $C^*$-algebras, modules, and $S$-equivariant
$KK$-theory are understood to be compatible.
The following result is the basis for this note.

\begin{theorem}  [Quigg and Sieben \cite{0992.46051}]    \label{theoremQuiggSieben}
There is an equivalence $\rho_{C^*}$ from the category
of $\calg$-$C^*$-algebras to the category of compatible $S$-Hilbert $C^*$-algebras.
It is given by $\rho_{C^*}(A,\alpha,\calg) = (A,\beta,S)$, where the
$\calg$-action $\alpha$ and $S$-action $\beta$ on $A$ determine each
other by
\begin{eqnarray}
(\beta_s(a))_{x} &=& 1_{\{x \in s s^*\}} \alpha_{x s} (a_{ s^* x s}),   \label{cstaraction}\\
\alpha_g(a_{s(g)})  &=& (\beta_s(a))_{r(g)}   \nonumber
\end{eqnarray}
for every $a \in A, x \in X$ and $g \in s \in S$.
Moreover, there is an isomorphism $\psi:A \widehat{\rtimes}_\beta
S \longrightarrow A \rtimes_\alpha \calg$ determined by
$(\psi(a \widehat \rtimes s))_g := 1_{\{ g \in s \}} a_{r(g)}$
for all $s \in S, a \in A_s$ and $g \in \calg$.
\end{theorem}

\begin{proof}
The maps and the isomorphism can be found in
Theorems 5.3, 6.2 and 7.1 in \cite{0992.46051}.
\end{proof}

By the last theorem we no longer distinguish between $\calg$- and compatible $S$-Hilbert $C^*$-algebras
because it is understood that we have both the $\calg$- and the $S$-action present if one of these actions is defined.

\begin{proposition}   \label{theroemRhoHilbertmodule}
There is an equivalence $\rho_{\rm H}$ from the category of $\calg$-Hilbert bimodules to the category of compatible $S$-Hilbert bimodules.
It is given by $\rho_{\rm H}(\pi,\cale,V,\calg) = (\pi,\cale,U,S)$, where
the $\calg$-action $V$ and $S$-action $U$ on $\cale$
determine each other by
\begin{eqnarray}
\big(U_s(\xi)\big)_{x} &=& 1_{\{x \in s s^*\}} V_{x s} \big(\xi_{ s^* x s} \big),  \label{hilbertaction}\\
V_g(\xi_{s(g)})  &=& (U_s(\xi))_{r(g)}   \label{hilbertgroupoidaction}
\end{eqnarray}
for $\xi \in \cale, x \in X$ and $g \in s \in S$. Actually,
$\rho_{\rm H} \rho_{\rm H}^{-1}= 1$ and $\rho_{\rm H}^{-1} \rho_{\rm H}=1$ are the identic functors,
%
and $\rho_{\rm H}$ respects the internal and external tensor
products.

%
\end{proposition}

\begin{proof}
Let us be given a $\calg$-action $V: s^* \cale \longrightarrow r^* \cale$, and define $U$ by (\ref{hilbertaction}).
Let $s \in S$ and $\xi \in \cale$. We aim to show that
$U_s(\xi)$ is an upper semi-continuous bundle over $X$, which would prove that $U_s(\xi)$ is in $\cale$.
Define a homeomorphism $\psi: s s^* \longrightarrow s$ by
$\psi(x) = x s$. Then $x=r(\psi(x))$, and so
$$\big(U_s(\xi)\big)_{r(\psi(x))} = 1_{\{r(\psi(x)) \in s s^*\}} V_{\psi(x)} \big(\xi_{s(\psi(x))}\big)$$
for all $x \in s s^*$ by (\ref{hilbertaction}). Consequently we have the identity
$$U_s(\xi) \circ r  =
V( \xi \circ s )$$
on the set $s$. Since
$\xi \circ s \in s^* \cale$, we have $V(\xi \circ s) \in r^* \cale$ by
\cite[Def. 3.3.1]{ParaviciniThesis}. Thus $U_s(\xi) \circ r$ is
upper semi-continuous on the set $s$ since $V( \xi \circ s )$ is upper semi-continuous there.
As $r|_s$ is a homeomorphism, $U_s(\xi)$ is upper semi-continuous on $s s^*$. Since $U_s(\xi)$ is defined to be zero
outside of $s s^*$ by (\ref{hilbertaction}), and since $s s^*$ is a clopen set, $U_s(\xi)$ is upper semi-continuous on $X$.

We need to check that $U$ defines a compatible $S$-action according to \cite[Def. 3.3]{burgiDescent}
and Definition \ref{definitionCompatibleHilbert}.
We define $U_s^* := U_{s^*}$ for all $s \in S$.
Recall that the inner product on $\cale$ can be computed in the bundle picture by
$\big(\langle \xi, \eta
\rangle_\cale \big)_{x} = \langle \xi_x , \eta_x \rangle_{\cale_x}$
for all $\xi, \eta \in \cale$ and $x \in X$.
Let $s \in S$ and $x \in X$. By (\ref{cstaraction}) and (\ref{hilbertaction})
we have
\begin{eqnarray*}
&& (\langle U_s(\xi),\eta  \rangle_\cale )_x = \langle ( U_s(\xi))_x,\eta_x \rangle_{\cale_{x}}
=  \langle 1_{\{x \in s s^*\}} V_{x s} (\xi_{ s^* x s} ), V_{x s} V_{s^* x}(\eta_x) \rangle_{x}   \\
&=& 1_{\{x \in s s^*\}} \alpha_{x s} \langle  \xi_{ s^* x s} ,   1_{\{x \in s s^*\}} V_{s^* x}(\eta_x) \rangle_{\cale_{s^* x s}}   \\
&=&  1_{\{x \in s s^*\}} \alpha_{x s} \langle  \xi_{ s^* x s} ,  1_{\{s^* x s \in s^* s\}} V_{s^* x s s^*}(\eta_{s s^* x s s^*}) \rangle_{\cale_{s^* x s}}   \\
&=& 1_{\{x \in s s^*\}} \alpha_{x s} \langle  \xi_{ s^* x s} ,  (U_{s^*}(\eta))_{s^* x s} \rangle_{\cale_{s^* x s}}\\
&=& (\beta_s \langle \xi , U_{s^*}(\eta) \rangle_{\cale})_x .  
\end{eqnarray*}
Since $x$ was arbitrary we have $\langle U_s(\xi),\eta  \rangle_\cale = \beta_s \langle \xi , U_{s^*}(\eta) \rangle_{\cale}$
as required in \cite[Def. 3.3]{burgiDescent}.
We leave the remaining verification that $U$ is a compatible $S$-action to the reader.

In the other direction, let us now be given an $S$-action $U$ and define $V$ by
(\ref{hilbertgroupoidaction}).
We want to show that $V$ is a continuous field of linear maps $V:s^* \cale \longrightarrow r^* \cale$.
$V$ is bounded
as we see by the estimate (with (\ref{hilbertgroupoidaction}))
$$\|V_g(\xi_{r(g)})\| = \|U_s(\xi)_{r(g)}\| \le \|U_s(\xi)\| \le \|\xi\| \le (1+\varepsilon) \|\xi_{r(g)}\|$$
(here we have chosen $\xi \in \cale$ with $\|\xi\| \le (1+ \varepsilon)\|\xi_{r(g)}\|$).

By (\ref{hilbertgroupoidaction}) we have
$V(\xi \circ s)= U_s(\xi) \circ r \in r^* \cale$.
Since $V$ is a bounded familiy of maps mapping simple function into $r^* \cale$,
and the simple functions form a total set in the sense of \cite[Proposition 3.1.26]{ParaviciniThesis},
$V$ is a continuous field of linear maps by \cite[Proposition 3.1.30]{ParaviciniThesis}.

$V$ is an automorphism in the sense of Le Gall \cite{legall1999} as for $g \in s \in S$ we have
$$V_{g^{-1} g}(\xi_{s(g)}) = V_{s(g)}(\xi_{s(g)}) = (U_{s^* s}(\xi))_{r(s(g))} = (1_{s* s} \cdot \xi)_{s(g)}= \xi_{s(g)}$$
by (\ref{hilbertgroupoidaction}) and the definition of a compatible Hilbert module.
The verification of $V_g V_h = V_{gh}$ is also straightforward.
The last claims like $\rho_{\rm H}(\cale \otimes \calf) = \rho_{\rm H}(\cale) {\otimes}  \rho_{\rm H}(\calf)$
are also easy to check.
%
\end{proof}

By Proposition \ref{theroemRhoHilbertmodule} we shall omit notating $\rho_{\rm H}$ since it is understood that we have both
the $\calg$- and the compatible $S$-action present if one of these actions is defined.
Recall (see \cite[5.1.2]{ParaviciniThesis}) that for a continuous field
of bilinear maps, $\mu:\cale \times \calf \longrightarrow \caln$, of
$\calg$-Banach spaces $\cale,\calf,\caln$, there is a convolution
product $* : C_c(r^* \cale) \times C_c(r^* \calf) \longrightarrow
C_c(r^* \caln)$ defined by
$$(\xi * \eta)_g = \int_{\calg^{r(g)}} \mu(\xi_h,  h(\eta(h^{-1} g))) d \lambda^{r(g)} (h)$$
for $g \in \calg$.
%
{The groupoid descent homomorphism $j^\calg: KK^\calg(A,B)
\longrightarrow KK(A \rtimes \calg, B \rtimes \calg)$ is defined by
$j^\calg[\pi,\cale,T] =  [\pi', \overline{C_c(r^* \cale)}, T \otimes 1]$,
where $C_c(r^* \cale)$ is to be closed under the $B\rtimes \calg$-inner, i.e. $\overline{C_c(r^* B)}$-inner, product
which is defined by convolution.
Here, $\pi'$ acts also by convolution. }


\begin{theorem}   \label{theoremKKQuiggSieben}
Let $S$ be a Cantor inverse semigroup embedded in a groupoid $\calg$.

(a) For all $\calg$-$C^*$-algebras $A$ and $B$ there exists a group
isomorphism
$$\rho :KK^\calg(A,B) \longrightarrow  \widehat{KK^S}(A,B)$$
induced by the identity map on cycles. (To be precise, $\rho[\cale,T] = [\rho_{\rm H}(\cale),T]$ for $[\cale,T] \in \E^\calg(A,B)$.)

(b) $\rho$ respects the Kasparov product, that is,
$\rho( x \otimes_B y) = \rho(x) \otimes_{B} \rho(y)$
if the product $x \otimes_B y$ is defined for $x \in KK^\calg(A,B)$ and
$y \in KK^\calg(B,C)$.

(c) $\rho$ respects functoriality in $A$ and $B$,
that is, equivariant homomorphisms $f:A' \rightarrow A$ and $g : B
\rightarrow B'$ enjoy $f^* \rho = \rho f^*$ and $g_* \rho = \rho
g_*$.

(d) One has $\rho \tau_D^\calg = \widehat{\tau_{D}^S} \rho$ for every separable $\calg$-$C^*$-algebra $D$.

(e) $\rho$ respects the Kasparov cup-cap product.

(f) Define a descent homomorphism $\widehat{j^{S}} :=  \nu
j^{\calg} \rho^{-1}$ by the commuting diagram
$$
\begin{xy}
\xymatrix {
KK^\calg(A,B) \ar[d]_{j^\calg}  \ar[rr]^\rho   & &  \widehat{KK^{S}} (A, B)  \ar[d]^{\widehat{j^S}}\\
KK(A \rtimes \calg,B \rtimes  \calg)  \ar[rr]^\nu  & &  KK ( A \widehat{\rtimes } S,  B  \widehat{\rtimes} S).
}
\end{xy}
$$
%
Here, the isomorphism $\nu$ is understood to be induced by the isomorphism $\psi$ of
Theorem \ref{theoremQuiggSieben}.
On the level of cycles one then has
\begin{eqnarray}
\widehat{j^S}[\pi,\cale,T] &=& [\hat \pi, \cale \otimes_B (B \widehat \rtimes  S), T \otimes 1], \label{hatJS} \\
\hat \pi(a \widehat \rtimes s) \big(\xi \otimes (b \widehat \rtimes t) \big ) &:=&
\pi(a) U_s (\xi) \otimes \big(s(b) \widehat \rtimes s t \big)    \label{tildePiJS}
\end{eqnarray}
for all $s,t \in S, \xi \in \cale, a \in A_s$ and $b \in B_t$.

%


%
\end{theorem}

\begin{proof}
(a) Let $(\pi,\cale,T)$ be a cycle in $\widehat{\E^S}(A,B)$.
Define the $\calg$-action $V$ as in Proposition
\ref{theroemRhoHilbertmodule}. We need to show that this action $V$ defines a groupoid equivariant Kasparov cycle in the sense of
\cite{legall1999}. Since the cycle $(\pi,\cale,T) \in
\E^\calg(A,B)$ is evidently already an ordinary Kasparov cycle without $\calg$-action, by
\cite[Def. 5.2]{legall1999} we only need to check that for every $a \in r^* A$
$$Q_{a} := \big((r^*\pi)(a) \big) \big(V (s^* T) V^* - r^* T \big)$$
is an element of $\calk(r^* \cale)$.
Let $\xi \in \cale$ and
$a:= b \circ r$ for $b \in A$.
Let $s \in S$ and $g \in s$.
Then we have
\begin{equation}  \label{eq321}
(U_{s s^*} \xi)_{r(g)} = (1_{s s^*} \xi)_{r(g)} = \xi_{r(g)}.
\end{equation}
Using this we have with (\ref{hilbertgroupoidaction}) (applied to $g$ and $g^{-1}$)
%
%
\begin{eqnarray*}
\big(Q_a (\xi \circ r)\big)_g &=& \big(\big((r^*\pi)(b \circ r) \big) \big( V (s^* T) V^* - (r^* T) \big) (\xi \circ r) \big)_{g}\\
&=& \big(\pi_{r(g)}(b_{r(g)}) \big)  \big (V_g T_{s(g)}V_{g^{-1}}- T_{r(g)} \big ) \xi_{r(g)}\\
&\stackrel{(\ref{hilbertgroupoidaction})}{=}& \big(\pi(b) ( U_s T U_{s^*} - T ) \xi \big)_{r(g)}\\
&=& \big(\pi(b) ( U_s T U_{s^*} - T U_{s} U_{s^*}) \xi \big)_{r(g)}\\
&=& (k_{b,s} (\xi))_{r(g)} = \big ( ( k_{b,s} \circ r) ( \xi \circ r) \big )_g
\end{eqnarray*}
if we set
$$k_{b,s}:=\pi(b) ( U_s T U_{s^*} - T U_s U_{s^*}) \in \call(\cale).$$
So we have proved that on the local neighborhood $s$, $(Q_a-(k_{b,s}\circ r))(\xi \circ r) = 0$.
Since we have this identity only on the set $s$, we restrict the bundles over $\calg$ to bundles over $s$ in the next argument.
Since $\big(Q_a-(k_{b,s} \circ r)\big)|_s$ is a bounded continuous field of linear maps,
and such a map is already determined
on the total set of simple functions $\big(\xi \circ r\big)|_s \in (r^* \cale)|_s$ (cf. \cite[Proposition 3.1.30]{ParaviciniThesis}),
we obtain $\big(Q_a-(k_{b,s} \circ r) \big)|_s= 0$ by approximation.

By definition of a Kasparov cycle (\cite[Def. 3.6]{burgiDescent}), $k_{b,s}$ is in $\calk(\cale)$, and so $k_{b,s} \circ r \in r^* \calk(\cale)$.
Consequently, $Q_a|_s \in \big(r^* \calk(\cale)\big)|_s$.
Varying over all $s \in S$, we see that $(Q_a)_g \in \calk(\cale_{r(g)})$ for all $g \in \calg$.
By a canonical isomorphism $r^* \calk(\cale) \cong \calk(r^* \cale)$ by Le Gall \cite{legall1999}, we obtain $Q_a \in \calk(r^* \cale)$.

Note that $Q: r^* A \longrightarrow \call(r^* \cale): a \mapsto Q_a$ defines a bounded $C_0(\calg)$-linear map.
Now let us be given any $a \in r^* A$. Since the set of simple functions $A \circ r$ forms a total set in $r^* A$, $(Q_a)_{g}$ is approximated
by elements of the form $(Q_{b \circ r})_g \in \calk(\cale_{r(g)})$ ($b \in A$), and so $(Q_a)_g \in \calk(\cale_{r(g)})$.
Hence $Q_a \in r^* \calk(\cale)$. This proves that $(\pi,\cale,T)$ is a cycle in $\E^\calg(A,B)$.

Now assume that $(\pi,\cale,T)$ is a cycle in $\E^\calg(A,B)$.
We will use the above computations. Now $Q_a$ is an element of $\calk(\cale)$. By the analog argument as before we obtain
$\big(k_{b,s} \circ r \big)|_s \in \big(r^* \calk(\cale)\big)|_s$.
Consequently, ${k_{b,s}}({r(g)}) \in \calk(\cale_{r(g)})$.
Since $r|_s : s \rightarrow s s^*$ is a homeomorphism, by varying $g \in s$ we obtain $k_{b,s}(x) \in \calk(\cale_x)$ for all $x \in s s^*$.
Since $k_{b,s}$ vanishes outside of $s s^*$, we obtain $k_{b,s}(x) \in \calk(\cale_x)$ for all $x \in X$.
By the isomorphism $\calk(\cale) = \calk({\rm id}_X^* \cale) \cong {\rm id}_X^* \calk(\cale)$
we see that $k_{b,s} \in \calk(\cale)$.
We also have $[T,U_{s s^*}]=[T,1_{s s^*}] =0$.
This checks that $(\pi,\cale,T)$ is a cycle in $\widehat{\E^S}(A,B)$.
We have now also proved homotopy
invariance of $\rho$, as $\rho$ maps also 
$\E^\calg(A,B[0,1])$ to $\widehat{\E^S}(A,B[0,1])$ on the level of cycles.
%

(b) The claim is valid because the Kasparov products are essentially defined in the same way:
compare \cite[Def. 6.1]{legall1999} with \cite[Def.
19]{burgiSemimultiKK}. Also, the points (c) and (d) are easy as $f^*$, $g_*$ and the map $\tau$ are defined in the same way in groupoid and
inverse semigroup equivariant $KK$-theory. The same can be said for point (e) which follows directly from points (b) and (d).

(f)  Let us give a remark about the descent homomorphism for groupoids.
Note that $C_c(r^* \cale)$ is a bundle over $\calg$, and so $\overline{C_c(r^* \cale)}$ is one.
It is isomorphic to $\cale \otimes_B (B \rtimes \calg)$.
Here, $B$ acts on $B \rtimes
\calg$ by $(b_x)_{x \in X} \cdot (c_g)_{g \in \calg} =
(b_{r(g)} c_g)_{g \in \calg}$ for $b \in B, c \in r^* B$.
The tensor product $\cale \otimes_B (B \rtimes \calg)$ is also a $C_0(\calg)$-module, since $B \rtimes \calg$ is one.
Note that we can write
$$\cale \otimes_B (B \rtimes \calg) \cong \big( \cale_{r(g)} \otimes (B \rtimes \calg)_g \big)_{g \in \calg} \cong \big( \cale_{r(g)} \otimes B_{r(g)}\big)_{g \in \calg}$$
in the bundle picture by identifying $\xi \otimes \big((b_g)_{g \in \calg}\big)$ with
$\big(\xi_{r(g)} \otimes b_g \big)_{g \in \calg}$.


By using the formulas (\ref{hatJS}) and (\ref{tildePiJS}) for $\widehat{j^S}$, given $(\pi,\cale,T) \in \E^\calg(A,B)$ we
have
\begin{eqnarray}
j^\calg [\pi,\cale,T]  &=&  [\pi',\cale \otimes_B (B \rtimes \calg), T \otimes 1],\\
\nu^{-1} \widehat{j^S} \rho[\pi,\cale,T] &=& [\psi^{-1} \hat \pi, \cale \otimes_B \psi (B \widehat \rtimes S), T \otimes 1].  \label{eq223}
\end{eqnarray}
We are going to show that $\pi' = \psi^{-1} \hat \pi$,
where we identify $\cale \otimes_B (B \rtimes \calg)$ with $\cale \otimes_B \psi (B \widehat \rtimes S)$
without spelling out an isomorphism for it.
(In this sense the notation $\psi^{-1} \hat \pi$ in (\ref{eq223}) is not precise.)
Let $s,t \in S, a \in A_s$ and $b \in B_t$. By Theorem \ref{theoremQuiggSieben} we have
\begin{eqnarray}
&& \hat \pi ( a \widehat \rtimes s)( \xi \otimes \psi(b
\widehat \rtimes t))  \label{equ221}  \\  
\label{sconvform} &=&  \hat \pi \big( \psi^{-1} (1_{\{g \in s \}}
a_{r(g)})_{g \in \calg} \big) \big ( (1_{\{g \in
t\}} \xi_{r(g)}  \otimes b_{r(g)})_{g \in \calg} \big).
\end{eqnarray}
On the other hand, by (\ref{tildePiJS}), (\ref{cstaraction}) and (\ref{hilbertaction}), (\ref{equ221}) becomes
\begin{eqnarray}
&& \pi(a) U_s (\xi) \otimes \psi \big(s(b) \widehat \rtimes s t \big)\\
\nonumber &=&
 \pi \big ( (a_x)_{x \in X} \big )  \Big (1_{\{x \in s s^* \}}
 V_{x s}(\xi_{s^* x s}) \Big)_{x \in X} \otimes  \Big (1_{\{g\in st\}} s(b)_{r(g)}\Big)_{g \in \calg} \\
 \label{someform}
 &=&  \Big ( 1_{\{g \in s t \}} 1_{\{r(g) \in s s^*\}}
\pi(a_{r(g)}) V_{r(g) s} (\xi_{s^* r(g)
 s}) \otimes \alpha_{r(g) s}(b_{s^* r(g) s}) \Big )_{g \in \calg}.
%
\end{eqnarray}
Computing $\pi'$ by convolution we get
\begin{eqnarray} 
&&  \pi' \big( (1_{\{g \in s \}}
a_{r(g)})_{g \in \calg} \big) \big ( (1_{\{g \in
t\}} \xi_{r(g)}  \otimes b_{r(g)})_{g \in \calg} \big)  \label{sconvform8} \\
&=&
\Big ( \int_{\calg^{r(g)}} \pi(1_{\{h \in s\}} a_{r(h)}) V_h \big
( 1_{\{ h^{-1} g \in t\}} \xi_{r(h^{-1} g)}    \big )  \label{sconvform2}  \\
&& \qquad \qquad  \otimes \alpha_h (b_{r(h^{-1} g)})  d
\lambda^{r(g)}(h)  \Big )_{g \in \calg}.  \nonumber
%
\end{eqnarray}
Since in (\ref{sconvform2}) there appears the expression $1_{\{h \in s\}} 1_{\{ h^{-1} g \in t\}}$,
the function under the integral vanishes unless $h = r(g) s$.
(Because $g \in s t$, and so
$g = (r(g) s) (t s(g))$, and this is the only possible
decomposition in $s$ and $t$.)
Recalling that $\lambda^{r(g)}$ is chosen to
be the counting measure, (\ref{sconvform2}) is exactly
(\ref{someform}). 
Consequently, (\ref{sconvform}) is (\ref{sconvform2}) and so $\pi' = \psi^{-1} \hat \pi$.
\end{proof}

\begin{corollary}   \label{corollaryBCmapSieben}
Let $S$ be a Cantor inverse semigroup embedded in a groupoid $\calg$ and $A$ a compatible $S$-Hilbert $C^*$-algebra. Then there exists a Baum--Connes map
\begin{eqnarray}
&& \widehat{\mu^S_A} : \lim_{Y \subseteq \underline E \calg} \widehat{KK^S} (C_0(Y),A) \longrightarrow K(A \widehat \rtimes S) \label{firstBCmap}.
\end{eqnarray}
\end{corollary}

\begin{proof}
This Baum--Connes map is immediately obtained by applying
the isomorphism $\rho$ of Theorem \ref{theoremKKQuiggSieben} and the isomorphism $\psi$ of Theorem \ref{theoremQuiggSieben}
to
Tu's \cite{0932.19005} Baum--Connes map for groupoids, 
$$\mu^\calg_A : \lim_{Y \subseteq \underline E \calg} KK^\calg (C_0(Y),A) \longrightarrow K(A \rtimes \calg).$$
\end{proof}

%

%

\section{The expansion homomorphism}    \label{sectionExpansion}

In this section we assume that $S$ is a Cantor inverse semigroup embedded in Paterson's groupoid
$\calg_S$, which we suppose to be Hausdorff, associated to $S$, see Paterson \cite{paterson1999} and
Lemma \ref{lemmaPatersonsGroupoidCX}.


%
%
%



\begin{proposition} [Khoshkam and Skandalis \cite{1061.46047}]    \label{lemmaExpansionFunctCstar}
There is a covariant (so-called) expansion functor $\epsilon_{C^*}$ from
the category of $S$-Hilbert $C^*$-algebras to the category of compatible $S$-Hilbert $C^*$-algebras
given by
$\epsilon_{C^*}(A,\alpha,S) = (A \rtimes_{\alpha|_E} E, \beta,S)$,
where the $S$-action $\beta$ is defined by
$\beta_s(a \rtimes e) := \alpha_s(a) \rtimes s e s^*$
for $e \in E, a \in A_e$ and $s \in S$.
For a morphism $\pi:A \rightarrow B$ between $S$-Hilbert $C^*$-algebras one sets
$\epsilon_{C^*}(\pi) = \pi \otimes 1$,
where $(\pi \otimes 1) (a \rtimes e) := \pi(a) \rtimes e$ for $e \in E$ and $a \in
A_e$.
%
\end{proposition}

\begin{proof}
The action $\beta$ was defined in Khoshkam--Skandalis \cite[Lemma
6.3]{1061.46047}. For the $C_0(X)$-structure see \cite[Proposition
5.13]{1061.46047}. It turns out that the action is compatible in the sense of Defintion \ref{DefCompatibleHilbertAlgebra}.
%
%
\end{proof}

%

\begin{theorem} [Khoshkam and Skandalis \cite{1061.46047}]   \label{theoremKoshkamSkandalis}
For a compatible $S$-Hilbert $C^*$-algebra $A$ define
$\beta$ as in Proposition \ref{lemmaExpansionFunctCstar}. Then there
are isomorphisms
%
$$
\begin{xy}
\xymatrix {
A \rtimes_\alpha S & (A \rtimes_{\alpha|_E} E)
\widehat \rtimes_\beta S \ar[l]_{\gamma} 
\ar[r]^{\psi} & (A \rtimes_{\alpha|_E} E) \rtimes \calg_S,\\
A \rtimes_{\alpha,{r}} S  \ar[rr]^{\mu} &&
(A \rtimes_{\alpha|_E} E) \rtimes_{r} \calg_S,
}
\end{xy}
$$
where
$\gamma((a \rtimes e) \widehat \rtimes s)  =  a \rtimes e s$
for
$s \in S, e \in E, e \le s s^*$ and $a \in  A_e$. (Here, $\psi$ denotes the isomorphism of
Quigg and Sieben, see Theorem \ref{theoremQuiggSieben}.)
\end{theorem}

\begin{proof}
For these results see \cite[Theorems 6.2 and 6.5]{1061.46047}.
\end{proof}

\begin{corollary}   \label{corollaryBCmapKhoshkamSkandalis}
Let $S$ be a Cantor inverse semigroup embedded in Paterson's universal groupoid $\calg_S$ and $A$ an $S$-Hilbert $C^*$-algebra. Then there exists a Baum--Connes map for Khoshkam and Skandalis' crossed product,
\begin{eqnarray}
&& \widehat{\mu^{S}_{A \rtimes E}} : \lim_{Y \subseteq \underline E \calg_S} \widehat{KK^S} (C_0(Y),A \rtimes E)
\longrightarrow K(A \rtimes S).
\label{secondBCmap}
\end{eqnarray}
\end{corollary}

\begin{proof}
This Baum--Connes map is immediately obtained by replacing $A$ by $A \rtimes_{\alpha|_E} E$ in the Baum--Connes map
(\ref{firstBCmap}) and using the isomorphism $\gamma$ of Theorem \ref{theoremKoshkamSkandalis}.
\end{proof}

\begin{definition}   \label{defintionExpansionHilbertmod}
{\rm For a non-compatible $S$-Hilbert $B$-module $\cale$ we more shortly denote the $S$-Hilbert module tensor product
$\cale \otimes_B (B \rtimes E)$ by $\cale \rtimes E$.
It is endowed with the diagonal
$S$-action $\overline U = U \otimes \beta$, that is,
$$\overline U_s \big(\xi \otimes (b \rtimes e)\big ) = U_s(\xi) \otimes \big(s(b) \rtimes s e s^* \big)$$
for all $\xi \in \cale, e \in E$ and $b \in B_e$.
}
\end{definition}

We remark that $B$ acts here on $B \rtimes E$ by left multiplication. This multiplication
is a (non-compatible) $S$-equivariant representation.
We set $T \otimes 1 := T \otimes {\rm id}_{B
\rtimes E} \in \call(\cale \rtimes E)$ for $T \in \call(\cale)$.

\begin{lemma}   \label{lemmaExpansionBimodule}
There is an expansion functor $\epsilon_{\rm H}$ within the categories of (non-compatible)
$S$-Hilbert bimodules (morphisms in this category play no role here). It is
given by $\epsilon_{\rm H}(\pi,\cale,U) = (\overline \pi,\cale \rtimes E,\overline
U)$ (see Definition \ref{defintionExpansionHilbertmod}), where
$\overline \pi: A \rtimes E \longrightarrow \call(\cale \rtimes E)$
is the compatible $S$-equivariant map defined by
$\overline \pi (a \rtimes e) := (\pi(a) \otimes 1)
\overline U_e$
for $e \in E$ and $a \in A_e$. Here $A \rtimes E$ means $\epsilon_{C^*}(A)$.
The functor $\epsilon_{\rm H}$ maps compatible Hilbert bimodules to compatible Hilbert bimodules.
\end{lemma}

\begin{proof}
We are going to show that $\overline \pi$ is compatibly equivariant and demonstrate the formula $\overline \pi(s(a \rtimes e)) = \overline U_s \pi(a \rtimes e) \overline U_s^*$.
Let $e,f \in E, a \in A_e, b \in B_f, \xi \in \cale$ and $s \in S$. We have
\begin{eqnarray*}
\overline \pi(s(a \rtimes e))(\xi \otimes (b \rtimes f)) &=& \pi(s(a) \rtimes s e s^*) (\xi \otimes (b \rtimes f))\\
&=& \pi(s(a)) U_{s e s^*}(\xi) \otimes (s e s^*(b) \rtimes s e s^* f),\\
\overline U_s \overline \pi(a \rtimes e) \overline U_{s^*} (\xi \otimes (b \rtimes f)) &=&
\overline U_s \pi(a \rtimes e) \big ( U_{s^*}(\xi) \otimes (s^*(b) \rtimes s^* f s) \big)\\
&=& U_s \pi(a) U_{e s^*}(\xi) \otimes \big( s e s^*(b) \rtimes s e s^* f s s^* \big)\\
&=& \pi(s(a)) U_{s e s^*}(\xi) \otimes (s e s^*(b) \rtimes s e s^* f).
\end{eqnarray*}
So $\epsilon_{\rm H}$ is a well defined functor.
If $\cale$ is compatible then the tensor product $\cale \otimes_B (B \rtimes E)$ is a compatible Hilbert $(A \rtimes E,B \rtimes E)$-bimodule
because $\overline \pi$ is compatible and the module multiplication is compatible as we have
\begin{eqnarray*}
\xi \otimes (b \rtimes e) \cdot k(c \rtimes f) &=& \xi \otimes k (b \rtimes e) \cdot (c \rtimes f)\\
 &=& U_k(\xi) \otimes k(b \times e) \cdot (c \times f)
\end{eqnarray*}
for all $\xi \in \cale, e,f,k \in E, b \in B_e$ and $c \in B_f$.
\end{proof}


\begin{lemma}   \label{lemmaExpasionTensorproducts}
Let $\cale$ and $\calf$ be compatible $S$-Hilbert bimodules.
Then there is a canonical isomorphism in the category of compatible $S$-Hilbert bimodules
\begin{eqnarray*}
(\cale \rtimes E) \otimes_{B \rtimes E} (\calf \rtimes E)  & \cong &
(\cale \otimes_B \calf) \rtimes E
\end{eqnarray*}
\end{lemma}

\begin{proof}
Let $\cale$ and $\calf$ modules over $B$ and $C$, respectively,
with $S$-actions denoted by $U$ and $V$, respectively.
$B \rtimes E$ acts on $\calf \rtimes E$ according
to Lemma \ref{lemmaExpansionBimodule}. 
The
transformation is given by
\begin{eqnarray*}
&& (\xi \otimes b \rtimes e) \otimes (\eta \otimes c \rtimes f)   \longmapsto
\big (\xi b \otimes V_{ef}(\eta) \big ) \otimes ef(c) \rtimes e f
\end{eqnarray*}
for $e,f \in E, \xi \in \cale, \eta \in \calf,  b \in B_e$ and $c \in C_f$.
We leave the straightforward detailed verification to the reader.
%
%
%
\end{proof}



{In \cite{burgiDescent}, a descent homomorphism
$j^S: KK^S(A,B) \longrightarrow KK (A \rtimes S,B \rtimes S)$
is defined by mapping a cycle $(\pi,\cale, T)$ to the cycle
$(\tilde \pi, \cale \otimes_B (B \rtimes S), T \otimes 1)$, where
$\tilde \pi (a \rtimes s) := (\pi(a) \otimes 1) (U_s \otimes V_s)$,
and $V_s(b \rtimes t) := s(b) \rtimes s t$ for $s, t \in S, a \in A_s$ and $b \in B_t$. }
By a slight abuse of notation let us denote
the composition of the canonical map $\widehat{KK^S}(A,B) \longrightarrow KK^S(A,B)$ (induced by the identity on cycles)
with the decent homomorphism $j^S$ also by $j^S$, that is, we have a descent homomorphism
$j^S: \widehat{KK^S}(A,B) \longrightarrow KK(A \rtimes S,B \rtimes S)$.

%


\begin{theorem}   \label{theoremExpansionHomomorphismHat}

Suppose that $S$ is the Cantor inverse semigroup of Paterson's groupoid $\calg_S$.
Let $A$ and $B$ be compatible $S$-Hilbert $C^*$-algebras.

(a) Then there exists an expansion group homomorphism
$$\epsilon : \widehat {KK^S} (A,B) \longrightarrow \widehat{KK^{S}}(A \rtimes E,B \rtimes E)$$
defined by $\epsilon[ \pi,\cale,T ]
=[\epsilon_{\rm H}(\pi,\cale),T \otimes 1] =[\overline \pi, \cale
\rtimes E, T \otimes 1]$ for $(\pi,\cale,T) \in \E^S(A,B)$.

(b) If $S$ has a unit then the expansion homomorphism respects
the intersection product, that is,
$\epsilon( x \otimes_B y) = \epsilon(x) \otimes_{B \rtimes E} \epsilon(y)$
if the product $x \otimes_B y$ is defined ($x \in \widehat{KK^S}(A,B), y \in
\widehat{KK^S}(B,C)$).

(c) $\epsilon$ respects functoriality in $A$ and $B$, i.e.
equivariant homomorphisms $f:A' \rightarrow A$ and $g : B
\rightarrow B'$ enjoy $(f \otimes 1)^* \epsilon = \epsilon f^*$
and $(g \otimes 1)_* \epsilon = \epsilon g_*$.

(d) $\epsilon$ intertwines Sieben's crossed product decent homomorphism and Khoshkam--Skandalis' crossed product descent homomorphism in the following sense: There
is a commuting diagram
$$
\begin{xy}
\xymatrix {
\widehat{KK^S}(A,B) \ar[d]_{j^S}  \ar[rr]^\epsilon   & &  \widehat{KK^{S}} (A \rtimes E, B \rtimes E)  \ar[d]^{\widehat{j^S}}\\
KK(A \rtimes S,B \rtimes  S)  \ar[rr]^\nu  & &  KK ( (A \rtimes E) \widehat{\rtimes } S, (B \rtimes E ) \widehat{\rtimes} S).
}
\end{xy}
$$
%
The bottom isomorphism $\nu$ is induced by $\gamma$ of Theorem
\ref{theoremKoshkamSkandalis}.
\end{theorem}

\begin{proof}
(a) As remarked above, we have a descent homomorphism $j^E:\widehat{KK^S}(A,B)
\longrightarrow KK(A \rtimes E,B \rtimes E)$.
Ignoring the $S$-action in the image of $\epsilon$, $\epsilon$ is defined on the level of cycles like the descent
homomorphism $j^E$.
So the cycle of $\epsilon(x)$ is in $\E(A \rtimes E,B \rtimes E)$.
That is why we need only examine the $S$-action.
In view of Lemma \ref{lemmaExpansionBimodule} it remains to examine the elements $U_s T U_s^* - U_s U_s^* T$
and $[U_s U_s^*, T]$ appearing in the definition of an $S$-equivariant Kasparov cycle in \cite[Def. 3.6]{burgiDescent}. 
More precisely, we have to check that $\overline
\pi(a) [\overline U_e, T \otimes 1]$ and $\overline \pi(a) (\overline U_s (T
\otimes 1) \overline U_{s^*} - (T \otimes 1) \overline U_{s s^*})$ are
compact operators in $\calk(\cale \otimes_B (B \rtimes E) )$ for all $a \in A,e \in E$ and $s \in S$.
This follows by a small computation
from $\overline U_e (\calk(\cale) \otimes 1) \subseteq \calk(\cale
\otimes_B (B \rtimes E))$ ($\forall e \in E$) which appears in the proof of \cite[Theorem 13.4]{burgiDescent}.

(b)
The claim is correct if we regard $\epsilon$ only as a descent homomorphism $\epsilon: KK^S(A,B) \longrightarrow KK(A \rtimes E,B \rtimes B)$
without the $S$-action in the image of $\epsilon$ by \cite[Theorem 13.4]{burgiDescent}.
The argument in \cite[Theorem 13.4]{burgiDescent} uses the isomorphism involving the bimodule $\cale_{12} = \cale_1 \otimes_B \cale_2$ stated in Lemma \ref{lemmaExpasionTensorproducts}.
Hence by extending the argument in \cite[Theorem 13.4]{burgiDescent}
by the facts that this lemma is also on isomorphism for the $S$-action, and that $\cale_{12}$ is also
a compatible bimodule when $\cale_1$ and $\cale_2$ are compatible, one gets that also $\epsilon$ respects the Kasparov product.

(c)
This can be checked by using Lemma
\ref{lemmaExpasionTensorproducts} and we ask the reader to elaborate the details.
(d)
We have
\begin{eqnarray*}
{j^S} [\pi,\cale,T] &=& [\tilde \pi, \cale \otimes_B (B \rtimes S), T \otimes 1],\\
\widehat{j^S} \epsilon [\pi,\cale,T] &=& \big[\hat{\overline \pi}, \big ( \cale \otimes_B (B \rtimes E) \big )
\otimes_{B \rtimes E} \big((B \rtimes E) \widehat \rtimes S\big), T \otimes 1 \otimes 1 \big]\\
 &=& \big [\hat{\overline \pi}, \cale \otimes_B \big((B \rtimes E) \widehat \rtimes S\big), T \otimes 1 \big]\\
 &=& [\hat{\overline \pi}, \cale \otimes_B (B \rtimes S), T \otimes 1 ],
\end{eqnarray*}
where we have used $\gamma$ of Theorem \ref{theoremKoshkamSkandalis}.
(Sloppily we identify the appearing $C^*$-algebras under $\gamma$ and do not spell out this isomorphism.)
We are going to show that $\tilde \pi = \hat{\overline \pi}$.
Let $e,f \in E, a\in A_e, b\in B_f^+$ be positive and $s,t \in S$. We have
\begin{eqnarray}
\tilde \pi \big( \gamma ((a \rtimes e) \widehat \rtimes s ) \big)\big(\xi \otimes
\gamma  ((b \rtimes f) \widehat \rtimes t ) \big) &=&
\tilde \pi (a \rtimes e s) (\xi \otimes
(b \rtimes f t)) \nonumber \\
&=& \pi(a) U_{es}(\xi ) \otimes (es(b) \rtimes es ft).  \label{eq400}
\end{eqnarray}
On the other hand we have
\begin{eqnarray}
&& \hat{\overline \pi} \big( (a \rtimes e) \widehat \rtimes s \big)
\big(\xi \otimes (\sqrt{b} \rtimes f) \otimes ( (\sqrt{b} \rtimes f) \widehat \rtimes t) \big)  \nonumber \\
&=& \overline \pi (a \rtimes e) \big (U_s (\xi) \otimes
(s(\sqrt{b}) \rtimes s fs^*) \big ) \otimes \big( (s(\sqrt{b}) \rtimes s f s^*) \widehat \rtimes s t) \big)  \nonumber \\
&=& \big (\pi (a) U_e U_s (\xi) \otimes
(e s(\sqrt{b}) \rtimes e s fs^*) \big ) \otimes \big( (s(\sqrt{b}) \rtimes s f s^*) \widehat \rtimes s t) \big)  \nonumber \\
&\cong& \pi (a) U_{e s} (\xi) \otimes
(e s(\sqrt{b}) \cdot e s fs^* s(\sqrt{b}) \rtimes e s fs^* s f s^*  s t) \nonumber \\
&=& \pi (a) U_{e s} (\xi) \otimes  
(e s(b) \rtimes e s f t).  \label{eq401}
\end{eqnarray}
Since (\ref{eq400}) and (\ref{eq401}) are equal, we get $\tilde \pi = \hat{\overline \pi}$,
and so $j^S = \widehat{j^S} \epsilon$.
\end{proof}

As the proofs of (a)-(c) in the last theorem do not involve the embedding of $S$ in a groupoid at all, 
for every inverse semigroup $S$
there exists also an expansion homomorphism $\epsilon':KK^S(A,B) \longrightarrow KK^S(A\rtimes E,B\rtimes B)$
similarly as in the last theorem (without hat and $\calg_S$ need also not be Hausdorff).
This homomorphism respects the Kasparov product and functoriality.

%


\begin{definition}    \label{defReducedDescent}
{\rm
One may define a descent homomorphism $j^S_r:\widehat{KK^S} (A,B) \longrightarrow KK( A  \rtimes_{r} S, B  \rtimes_{r} S)$
for Khoshkam and Skandalis' reduced crossed product \cite{1061.46047}
by setting $j^S_r := \tau j_r^{\calg_S} \rho^{-1} \epsilon$. Here, $\tau$ denotes the isomorphism induced by $\mu$ of
Theorem \ref{theoremKoshkamSkandalis}.
}
\end{definition}

\bibliographystyle{plain}
\bibliography{references}

\end{document}